\documentclass[11pt]{article}

\usepackage{amsmath, amssymb, amsthm, verbatim,enumerate,bbm,color} 

\usepackage{mathrsfs}

\usepackage[backref,colorlinks,citecolor=blue,bookmarks=true]{hyperref}

\title{Exact Limit Theorems for Restricted Integer Partitions}

\author{Asaf Cohen\thanks{School of Mathematical Sciences, Tel Aviv University, Tel Aviv, 6997801, Israel. Supported in part by ISF Grant 1145/18.} \and Asaf Shapira \thanks{
School of Mathematics, Tel Aviv University, Tel Aviv 69978, Israel.
Email: asafico$@$tau.ac.il. Supported in part by ISF Grant 1028/16, ERC Consolidator Grant 863438 and NSF-BSF Grant 20196.}}

\date{\today}
\allowdisplaybreaks

\usepackage{enumitem}
\newlist{casesp}{enumerate}{3} 
\setlist[casesp]{align=left, 
                 listparindent=\parindent, 
                 parsep=\parskip, 
                 font=\normalfont\bfseries, 
                 leftmargin=0pt, 
                 labelwidth=0pt, 
                 itemindent=.4em,labelsep=.4em, 
                 partopsep=0pt, 
                 }
\setlist[casesp,1]{label=Case~\arabic*:,ref=\arabic*}
\setlist[casesp,2]{label=Case~\thecasespi.\roman*:,ref=\thecasespi.\roman*}
\setlist[casesp,3]{label=Case~\thecasespii.\alph*:,ref=\thecasespii.\alph*}

\theoremstyle{plain}
\newtheorem{theorem}{Theorem}[section]
\newtheorem{lemma}[theorem]{Lemma}
\newtheorem{claim}[theorem]{Claim}

\makeatletter
\def\moverlay{\mathpalette\mov@rlay}
\def\mov@rlay#1#2{\leavevmode\vtop{%
   \baselineskip\z@skip \lineskiplimit-\maxdimen
   \ialign{\hfil$\m@th#1##$\hfil\cr#2\crcr}}}
\newcommand{\charfusion}[3][\mathord]{
    #1{\ifx#1\mathop\vphantom{#2}\fi
        \mathpalette\mov@rlay{#2\cr#3}
      }
    \ifx#1\mathop\expandafter\displaylimits\fi}
\makeatother

\makeatletter
\renewenvironment{proof}[1][\proofname]
{\par\pushQED{\qed}
	\normalfont\topsep6\p@\@plus6\p@\relax\trivlist
	\item[\hskip\labelsep\bfseries#1\@addpunct{.}]
	\ignorespaces}
{\popQED\endtrivlist\@endpefalse}
\makeatother

\RequirePackage{color}\definecolor{RED}{rgb}{1,0,0}\definecolor{BLUE}{rgb}{0,0,1} 

\begin{document}
\date{}
\maketitle

\begin{abstract}
For a set of positive integers $A$, let $p_A(n)$ denote the number of ways to write $n$ as a sum
of integers from $A$, and let $p(n)$ denote the usual partition function. In the early 40s, Erd\H{o}s
extended the classical Hardy--Ramanujan formula for $p(n)$ by showing that $A$ has density $\alpha$
if and only if $\log p_A(n) \sim \log p(\alpha n)$. Nathanson asked if Erd\H{o}s's theorem holds also with respect to $A$'s lower density, namely,
whether $A$ has lower-density $\alpha$ if and only if $\log p_A(n) / \log p(\alpha n)$ has lower limit $1$.
We answer this question negatively by constructing, for every $\alpha > 0$, a set of integers $A$ of lower density $\alpha$, satisfying
$$
\liminf_{n \rightarrow \infty} \frac{\log p_A(n)}{\log p(\alpha n)} \geq \left(\frac{\sqrt{6}}{\pi}-o_{\alpha}(1)\right)\log(1/\alpha)\;.
$$
We further show that the above bound is best possible (up to the $o_\alpha(1)$ term), thus determining the exact extremal relation between
the lower density of a set of integers and the lower limit of its partition function.
We also prove an analogous theorem with respect to the upper density of a set of integers, answering another question of Nathanson.


\end{abstract}

\section{Introduction}\label{secintro}

A partition of an integer $n$ is a sequence of positive integers $a_1 \leq a_2 \leq \ldots $ whose sum is $n$.
The classical partition function $p(n)$ denotes the number of partitions of $n$. More generally, for a set of positive integers $A$,
we denote by $p_A(n)$ the number of partitions of $n$ using integers taken from $A$. The study of various properties
of these restricted partition functions is amongst the oldest topics in mathematics. Some classical examples are Euler's Pentagonal Numbers Theorem and the Rogers--Ramanujan identities.
The reader is referred to \cite{Andrews72,Andrews1976,Andrews1976-1} for a more thorough background on this topic.
Our goal in this paper is to obtain asymptotic estimates for such restricted partition functions.
Arguably, the most well known result of this type is the classical Hardy--Ramanujan
formula \cite{HR1} (discovered independently by Uspensky \cite{Uspensky1920}) stating\footnote{The results of \cite{HR1,Uspensky1920} actually give much more accurate asymptotic estimates for $p(n)$.} that

\begin{equation}\label{equation - Hardy Ramanujan}
    p(n)\sim\frac{1}{4n\sqrt{3}}e^{\pi \sqrt{\frac{2n}{3}}}\;.
\end{equation}

Following \cite{HR1}, asymptotic estimates for $p_A(n)$ were obtained for various sets $A$. For example, already Hardy and Ramanujan \cite{HR2} obtained bounds analogous to \eqref{equation - Hardy Ramanujan} when $A$ is the set of primes, when $A$ is the set of odd integers, and when $A$ is the set of $k^{th}$ powers of positive integers. Szekeres \cite{Szekeres1,Szekeres2} obtained tight asymptotic bounds for partitions
avoiding large numbers, that is, when\footnote{We use $[a,b]$ to denote the integers $\{a,\ldots,b\}$. We also use $[a]$ to denote the integers $\{1,\ldots,a\}$.} $A=[1,m(n)]$ for various functions $m(n)$, see also \cite{Canfield,Romik}. In the other direction,
Diximier and Nicolas \cite{DN} studied partitions avoiding small integers, namely, when $A=[m(n),n]$ for various functions $m(n)$, see also \cite{Mosaki,MosakiNicolasSarkozy}. Finally, Nathanson \cite{Nat2000b} and Erd\H{o}s and Lehner \cite {Erdos2} studied the case of $A$ of fixed size.


In a remarkable paper from the early 40’s, Erd\H{o}s \cite{Erdos} gave an elementary proof of a slightly weaker version of (\ref{equation - Hardy Ramanujan}). He further extended (\ref{equation - Hardy Ramanujan}) by showing that if $A$ is a set of density $\alpha$ with\footnote{Note that if $\gcd(A)=d > 1$ then trivially $p_A(n)=0$ whenever $n$ is not divisible by $d$. In our proofs it will be very easy to guarantee that $\gcd(A)=1$ since all the sets $A$ we construct contain two consecutive integers.} $\gcd(A)=1$,
then $p_A(n)$ behaves like $p(\alpha n)$, more precisely
\footnote{For simplicity, we frequently remove floor/ceiling notation when they make no real difference. For example, in (\ref{eqErdos1}) the $\alpha n$ should really be $\lfloor \alpha n \rfloor$. Also, throughout the paper, logarithms are natural unless stated otherwise.}
\begin{equation}\label{eqErdos1}
\lim_{n \to \infty}\frac{\log p_A(n)}{\log p(\alpha n)} = \lim_{n \to \infty}\frac{\log p_A(n)}{\pi \sqrt{2\alpha n/3}} = 1\;.
\end{equation}
More surprisingly, using the Hardy--Littlewood Tauberian Theorem \cite{HardyLittlewood1914}, Erd\H{o}s proved\footnote{A remark for the history buff: this result was actually stated as an open problem in the preliminary version of \cite{Erdos} and then sketched in the published version.
A full proof was given by Nathanson \cite{Nat2000a}, see also \cite{Nat2000c}.}
an ``inverse theorem’’, stating that if $A$ satisfies (\ref{eqErdos1}) then $A$ has density $\alpha$.
Together, these two theorems imply that $A$ has density $\alpha$ if and only if (\ref{eqErdos1}) holds.
Other inverse theorems of this type were obtained in \cite{Freiman1955,Kohlbecker1958,Yang2001}.

Given Erd\H{o}s's theorem \cite{Erdos}, it is natural to ask if the lower and upper densities of $A$ also uniquely determine
the lower and upper limits of $p_A(n)$. That is, whether $A$ has lower density $\alpha$ (respectively, upper density $\beta$) if and only if $\liminf_{n \to \infty} \frac{\log p_A(n)}{\log p(\alpha n)} =1$ (respectively, $\limsup_{n \to \infty} \frac{\log p_A(n)}{\log p(\beta n)}=1$).
This question was first raised by Nathanson \cite{Nat2000a}, who further
proved the following theorem, which is a strengthening\footnote{Note that when $\alpha=\beta$ (i.e. when $A$ has density $\alpha$) this theorem is equivalent to Erd\H{o}s's first theorem.} of the first theorem of Erd\H{o}s mentioned above.

\begin{theorem}[Nathanson \cite{Nat2000a}] Suppose $A$ is a set of integers with $\gcd(A)=1$ of lower density $\alpha$ and upper density $\beta$. Then
\begin{equation}\label{Nathan1}
\liminf_{n \to \infty} \frac{\log p_A(n)}{\log p(\alpha n)} \geq 1 ~~~~~\mbox{and}~~~~~ \limsup_{n \to \infty} \frac{\log p_A(n)}{\log p(\beta n)} \leq 1\;.
\end{equation}
\end{theorem}

Nathanson \cite{Nat2000a} asked if the above inequalities are in fact equalities, namely, whether one can prove inverse theorems (in the sense of Erd\H{o}s’s inverse theorem mentioned above) with respect to the lower and upper densities of $A$. Our main qualitative results in this paper are
that (perhaps unexpectedly) the answers to both questions are negative. As we explain below,
we moreover give optimal quantitative results relating the lower/upper densities of $A$ and the lower/upper limits of $p_A(n)$.

Our first result deals with the upper density of $A$. It shows that for all small enough $\beta$ there is a set $A$ of upper density $\beta$ so that
$\limsup_{n \to \infty} \frac{\log p_A(n)}{\log p(\beta n)}<1$.
We in fact determine precisely how small can this upper limit be.
This, in particular, implies a negative answer to Nathanson’s question for all small enough $\beta$.

\begin{theorem}\label{theoupper} For every $0<\beta < 1$ there is a set of integers $A$ with $\gcd(A)=1$ of upper density $\beta$ satisfying
\begin{equation}\label{eqTheo1}
\limsup_{n\to \infty} \frac{\log p_A(n)}{\log p(\beta n)} \leq \frac{\sqrt{6}\log 2}{\pi} + o_\beta(1)\;.
\end{equation}
Furthermore, the constant above is best possible. Namely, any $A$ of upper density $\beta$ satisfies
\begin{equation}\label{eqTheo2}
\limsup_{n \to \infty} \frac{\log p_A(n)}{\log p(\beta n)} \geq \frac{\sqrt{6}\log 2}{\pi} + o_\beta(1)\;.
\end{equation}
\end{theorem}

Our second and main result deals with the lower density of $A$. Contrary to the case of the upper density, if $A$ has lower density $\alpha $ then $\liminf_{n \to\infty} \frac{\log p_A(n)}{\log p(\alpha n)}$ cannot even be bounded
from above by an absolute constant. Again, this implies a negative answer to Nathanson’s question for all small enough $\alpha$.

\begin{theorem}\label{theolower} For every $0<\alpha < 1$ there is a set of integers $A$ of lower density $\alpha$ with $\gcd(A)=1$ satisfying
\begin{equation}\label{eqTheo3}
\liminf_{n \to\infty} \frac{\log p_A(n)}{\log p(\alpha n)} \geq (1-o_\alpha(1))\frac{\sqrt{6}}{\pi}\log(1/\alpha) \;.
\end{equation}
Furthermore, the above lower bound is best possible. Namely, any $A$ of lower density $\alpha$ satisfies
\begin{equation}\label{eqTheo4}
\liminf_{n \to \infty} \frac{\log p_A(n)}{\log p(\alpha n)} \leq (1+o_\alpha(1))\frac{\sqrt{6}}{\pi}\log(1/\alpha) \;.
\end{equation}
\end{theorem}

\paragraph{Proof and paper overview:}

The proof of Theorem \ref{theolower} appears in Section \ref{secproof 2} and the proof of Theorem \ref{theoupper} appears in Section \ref{secproof}.
All the proofs in this paper are elementary in the number theoretic sense \cite{Nat2000c}, that is, they rely on combinatorial/counting arguments and do not use complex analysis which is frequently used when studying partition functions.
We find it quite surprising that such elementary methods can yield the precise results stated in Theorems \ref{theoupper} and \ref{theolower}.
The results that are the most challenging to prove are those stated in (\ref{eqTheo1}) and (\ref{eqTheo3}). In both cases, the constructions of the sets $A$ are quite simple and rely on the following finitary intuition: if one has to choose a subset $A \subseteq [n]$ of size $\alpha n$ so as to maximize $p_A(m)$, one would choose $S=\{1,\ldots,\alpha n\}$, since small integers give more ``freedom''\footnote{See Lemma \ref{Lemma - the first part is always the best} where this intuition is formalized.}. Similarly, taking $S=\{n-\alpha n,\ldots,n\}$ would minimize $p_A(m)$.
The constructions of $A$ in both proofs are then an infinite variant of this finitary intuition.
While the constructions of the sets $A$ are simple, their analysis is quite
involved, relying among other things, on special cases of the results of Szekeres \cite{Szekeres1,Szekeres2} and Diximier--Nicolas \cite{DN} mentioned above. While the original proofs of these two results were highly non-elementary, we will provide short and self-contained proofs of the special cases we need in this paper, see Lemmas \ref{Theorem - szekeres} and \ref{Theorem - France}. The latter proof might be of independent interest.

\section{Proof of Theorem \ref{theolower}}\label{secproof 2}

We start with the proof of Theorem \ref{theolower} equation \eqref{eqTheo3}. To this end, we will first consider the ``easy’’ cases, handled by Lemma \ref{Lemma - main lemma liminf lower bound} below,
and then move on to consider the harder cases, which will be dealt with in the proof itself later on. This proof will require a certain amount of preparation which will be given after the proof of Lemma \ref{Lemma - main lemma liminf lower bound}.

In the next proof, as well as in the rest of this section, we will frequently use the basic inequalities $\left(\frac{n}{k}\right)^{k}\leq \binom{n}{k}\leq \left(\frac{en}{k}\right)^k$ and $(n/e)^{n}\leq n!\leq en(n/e)^n$.
Furthermore, throughout the paper we will frequently use the fact that there are $\binom{n+k-1}{k-1}$ solutions in nonnegative integers to the equation $\sum_{i=1}^{k}x_i=n$.

\begin{lemma}\label{Lemma - main lemma liminf lower bound}
For every $0< \alpha < 1$ there exists $n_1=n_1(\alpha)$ such that the following holds for every integer $n>n_1$. Letting $A=\{1\}\cup [n,\alpha n^2]$ we have the following for all $m\in [16\alpha n^2,\alpha n^4/16]$
\begin{equation}\label{lemma 3.1}
        \log p_{A}(m) \geq (2\log(1/\alpha)-8)\sqrt{\alpha m}\;.
\end{equation}
\end{lemma}

\begin{proof}
Let $0<\alpha <1$ and let $n$ be a positive integer with $n>n_1=n_1(\alpha)$ which will be specified later. Let $m\in[16 \alpha n^2,\alpha n^4]$. The proof splits into two cases depending on the value of $m$.
\begin{casesp}
\item
Assume $16 \alpha n^2\leq m \leq 4\alpha^3 n^4$. Our assumption in this case implies the following two inequalities:
\begin{align}
    2{\sqrt{\alpha m}}\left(n+{\sqrt{m/16\alpha}}\right) \leq m\;,\label{eq1 - lower density}\\
    \sqrt{m/16\alpha}\leq \alpha n^2/2\;.\label{eq2 - lower density}
\end{align}
Therefore, provided $n_1\geq 2/\alpha$ we can deduce from \eqref{eq2 - lower density} that
\begin{equation*}
    n+\sqrt{m/16\alpha}\leq \alpha n^2\;.
\end{equation*}
Setting $B=\{1\}\cup [n,n+\sqrt{m/16\alpha}]$ we infer that,
\[
    B \subseteq\{1\} \cup \left[ n,\alpha n^2\right]\subseteq A\;,
\]
implying that it is enough to prove a lower bound with respect to $p_B(m)$ in \eqref{lemma 3.1}.
We map each solution in nonnegative integers of the equation
\begin{equation} \label{equation - lower bound first equation}
    \sum _{k=0}^{\sqrt{ m/16\alpha}}x_k=2\sqrt{\alpha m}
\end{equation}
to a partition of $m$ with parts in $B$ as follows: if $(x_k)_{k=0}^{\sqrt{m/16\alpha}}$ is a solution in nonnegative integers of \eqref{equation - lower bound first equation} then for every $k$ we take the integer $n+k$ exactly $x_k$ times and finally take $1$ exactly $ m-\sum _{k=0}^{\sqrt{ m/16\alpha}}x_k=2\sqrt{\alpha m}$ times. This map is well defined as by \eqref{eq1 - lower density} and \eqref{equation - lower bound first equation} we have,
\[
    \sum _{k=0}^{\sqrt{m/16\alpha}}x_k (n+k)\leq(n+\sqrt{m/16\alpha})\sum_{k=0}^{\sqrt{m/16\alpha}}x_k=2{\sqrt{\alpha m}}(n+\sqrt{m/16\alpha})\leq m\;.
\]
Moreover, provided $n_1\geq 2$ the above map is an injection as $1\not \in [n,n+\sqrt{m/16\alpha}]$.
Combining the above observations we have
\begin{align*}
     p_{A}(m)&\geq p_{B}(m)\geq \binom{2\sqrt{\alpha m}+\sqrt{m/16\alpha}}{2\sqrt{\alpha m}-1}
     \geq \binom{\sqrt{m/16\alpha}}{2\sqrt{\alpha m}-1}\\
     &\geq \left(\frac{\sqrt{m/16\alpha}}{2\sqrt{\alpha m}-1}\right)^{2\sqrt{\alpha m}-1} \geq \left(\frac{1}{8\alpha }\right)^{2\sqrt{\alpha m}-1} \\
     &\geq \exp\left((2\log(1/\alpha)-8)\sqrt{\alpha m} \right)\;,
\end{align*}
where that last inequity holds provided $ n\geq 1/4\sqrt{\alpha}$.
\item
Assume $n^2/4\alpha \leq m \leq \alpha n^4/16$. Note that provided $n_1\geq 4/\alpha^2$ we have $n^2/4\alpha\leq 4\alpha ^3n^4$ for all $n>n_1$, hence Case 1 and Case 2 cover all $m\in[16\alpha n^2,\alpha n^4/16]$. Our assumption in this case implies the following two inequalities:
\begin{align}
    {\sqrt{m/16\alpha}}\left(n+{2\sqrt{\alpha m}}\right) \leq m\;,\label{eq3 - lower density}\\
    2{\sqrt{\alpha m}}\leq \alpha n^2/2\;.\label{eq4 - lower density}
\end{align}
Therefore, provided $n_1\geq 2/\alpha$ we can deduce from \eqref{eq4 - lower density} that
\begin{equation*}
    n+2\sqrt{\alpha m}\leq \alpha n^2\;.
\end{equation*}
Setting $B=\{1\}\cup [n,n+2\sqrt{\alpha m}]$ we deduce that,
\[
    B\subseteq \{1\}\cup \left[ n+1,\alpha n^2\right]\subseteq A\;.
\]
Similar to the first case, we may thus prove a lower bound for $p_B(m)$ in \eqref{lemma 3.1}.
We map each solution in nonnegative integers of the equation
\begin{equation}\label{equation - lower bound second equation}
    \sum _{k=0}^{2\sqrt{ \alpha m}}x_k=\sqrt{m/16\alpha}
\end{equation}
to a partition of $m$ with parts in $B$ as follows: if $(x_k)_{k=0}^{\sqrt{2\alpha m}}$ is a solution in nonnegative integers of \eqref{equation - lower bound second equation} then for every $k$ we take the integer $n+k$ exactly $x_k$ times and $1$ exactly $m-\sum _{k=0}^{2\sqrt{\alpha m}}x_k(n+k)$ times. This map is well defined as by \eqref{eq3 - lower density} and \eqref{equation - lower bound second equation} we have,
\[
    \sum _{k=0}^{2\sqrt{\alpha m}}x_k (n+k)\leq(n+2\sqrt{\alpha m}) \sum_{k=0}^{2\sqrt{\alpha m}}x_k=\sqrt{m/16\alpha}(n+2\sqrt{\alpha m})\leq m\;.
\]
Moreover, provided $n_1\geq 2$ this map is an injection as $1\not \in [n,n+\sqrt{2\alpha m}]$. Combining the above observations we have,
\begin{align*}
     p_{A}(m)&\geq p_{B}(m)\geq \binom{2\sqrt{\alpha m}+\sqrt{m/16\alpha}}{2\sqrt{\alpha m}}\geq \binom{\sqrt{m/16\alpha}}{2\sqrt{\alpha m}}\\
     &\geq \left(\frac{\sqrt{m/16\alpha}}{2\sqrt{\alpha m}}\right)^{2\sqrt{\alpha m}}= \left(\frac{1}{8\alpha}\right)^{2\sqrt{\alpha m}}
     \geq \exp\left((2\log(1/\alpha)-8)\sqrt{\alpha m}\right)\;. \qedhere
\end{align*}
\end{casesp}
\end{proof}

We will now prove several claims and lemmas which will be used in the proof of Theorem \ref{theolower} equation \eqref{eqTheo3}.
We start with the following very crude bound which will suffice for our purposes.

\begin{claim}\label{Claim - trivial bound}
    Suppose $n$ is positive integer and $A$ is a set of positive integers with $|A|=k$. Then
    \[
        p_{A}(n)\leq (n+1)^k\;.
    \]
\end{claim}

\begin{proof}
Since each partition of $n$ using integers from $A$ can contain each of these integers at most $n$ times, the number
of such partitions is clearly at most $(n+1)^k$.
\end{proof}

For positive integers $k,n$ we define $p_k(n)$ to be the number of ways to write $n$ as a sum of {\em exactly} $k$ nonnegative integers (without consideration of the ordering of the summands). The following three lemmas are folklore, and are proved here for the sake of completeness.

\begin{lemma}\label{Lemma - estimation for p_k}
Suppose $k,n$ are positive integers. Then,
\[
   \binom{n-1}{k-1}\leq k! \cdot p_k(n)\leq \binom{n+\binom{k}{2}-1}{k-1}\;.
\]
\end{lemma}

\begin{proof}
There are exactly $\binom{n-1}{k-1}$ {\em ordered} partitions of $n$ with $k$ positive parts. This implies the first inequality. To see the second, suppose $y_1 \leq y_2 \leq \ldots \leq y_k$ satisfy $\sum_{i=1}^{k}y_i=n$. Defining $x_i=y_i+i-1$ for all $i$, we have $\sum_{i=1}^{k}x_i=n+\binom{k}{2}$. As all $x_i$ are distinct, each permutation of $x_i$s give rise to a different \emph{ordered} solution to the equation $\sum_{i=1}^{k}z_i=n+\binom{k}{2}$ with nonnegative integers. This implies the second inequality.
\end{proof}

In the following lemma, as well as in the rest of the section, we use several times the notation $p_{[k]}(n)$. For clarity, we wish to emphasize that $p_{[k]}(n)$ stands for $p_{A}(n)$ where $A=[k]$.

\begin{lemma}\label{Lemma - at most k parts}
Suppose $k,n$ are positive integers. Then,
\[
    p_{[k]}(n)=p_{k}(n+k)\;.
\]
\end{lemma}

\begin{proof}
As it is well known, $p_{[k]}(n)$ is also the number of ways to write $n$ as a sum of at most $k$ integers. Let $(y_i)_{i=1}^{k}$ be a partition of $n+k$. Setting $x_i=y_i-1$ we obtain a partition of $n$ with at most $k$ parts. This process is invertible and therefore we obtain the assertion of the lemma.
\end{proof}

\begin{lemma}\label{Lemma - the first part is always the best}
Suppose $A$ is a set of positive integers with $|A|= k$. Then,
\[
    p_{A}(n)\leq p_{[k]}(n)\;.
\]
\end{lemma}

\begin{proof}
Let $f$ be the bijection between $A$ and $[k]$ defined by sending the $i^{th}$ largest integer of $A$ to $i$. We now define an injection $g$ between the partitions of $n$ with parts in $A$ and partitions of $n$ with parts in $[k]$. Given a partition $x=(x_i)_{i=0}^\ell$ of $n$ with parts in $A$, we define $g(x)$ to be the partition $y=(y_i)_{i=0}^{\ell+\ell'}$ where $\ell'=m-\sum_{i=0}^{\ell}f(x_i)$ and $y_i$ is defined to be $f(x_i)$ for all $0 \leq i\leq \ell$ and $1$ for all $\ell<i\leq \ell'$. To see that this is an injection let $x=(x_i)_{i=1}^{\ell_1},x'=(x')_{i=1}^{\ell_2}$ be two partitions of $n$ with all parts in $A$ and assume that $g(x)=g(x')$. Let $a$ the minimal integer in $A$. Since $f$ is a bijection each of the integers in $A$ besides $a$ must appear in $x$ and $x'$ the same number of times.
Since $\sum _{i=1}^{\ell_1}a_i=\sum _{i=1}^{\ell_2}b_i=n$, the integer $a$ appears the same number of times in $x$ and $x'$. This completes the proof.
\end{proof}

We now turn to prove the two key lemmas that will be used in the proof of (\ref{eqTheo3}).
The first is Lemma \ref{Theorem - szekeres} below.
We remark that this result can be derived (with some effort) from the more precise bound due to Szekeres \cite{Szekeres1,Szekeres2} (see also \cite{Canfield,Romik}), but the self contained elementary proof below is significantly simpler.

\begin{lemma}\label{Theorem - szekeres}
For every $0< \gamma < 1$ there exists $n_2=n_2(\gamma)$ such that for all $n>n_2$ we have
\[
\log p_{[\gamma \sqrt{n}]}(n) = (2\gamma \log(1/\gamma)+ \Theta(\gamma))\sqrt{n}\;.
\]
\end{lemma}

\begin{proof}
Let $0< \gamma <1$ be a real number and let $n$ be an integer with $n>n_2=n_2(\gamma)$ where $n_2$ will be specified later. Setting $k=\gamma \sqrt{n}$, Lemmas \ref{Lemma - estimation for p_k} and \ref{Lemma - at most k parts} implies that

\begin{equation}\label{equation - bound on p_k}
        \binom{n+\gamma \sqrt{n}-1}{\gamma \sqrt{n}-1}\leq (\gamma \sqrt{n})!\cdot p_{[\gamma \sqrt{n}]}(n)\leq \binom{n+\gamma \sqrt{n}+\binom{\gamma \sqrt{n}}{2}-1}{\gamma \sqrt{n}-1}\;.
\end{equation}
Therefore we have,
\begin{align*}
    p_{[\gamma\sqrt{n}]}(n) &\leq \frac{ \binom{n+\gamma \sqrt{n}+\binom{\gamma\sqrt{n}}{2}-1}{\gamma\sqrt{n}-1}}{(\gamma\sqrt{n})!} \leq \frac{\binom{2n}{\gamma \sqrt{n}}}{(\gamma \sqrt{n})!}\\
    &\leq \left(\frac{2e^2}{\gamma^2}\right)^{ \gamma\sqrt{n}} \leq e^{(2\gamma\log(1/\gamma)+4\gamma)\sqrt{n}}\;,
\end{align*}
where the second inequality holds provided $\gamma\sqrt{n_2}+\binom{\gamma \sqrt{n_2}}{2}-1\leq n_2$.
As to the lower bound we have,
\begin{align*}
    p_{[\gamma \sqrt{n}]}(n)&\geq \frac{\binom{n+\gamma \sqrt{n}-1}{\gamma \sqrt{n}-1}}{(\gamma \sqrt{n})!}\geq\left(\frac{\sqrt{n}}{\gamma}\right)^{\gamma \sqrt{n}-1}\left(\frac{e}{\gamma \sqrt{n}}\right)^{\gamma \sqrt{n}}\frac{1}{e\gamma \sqrt{n}}\\
                    &= \left(\frac{e}{\gamma^2}\right)^{\gamma \sqrt{n}}\frac{1 }{en}=e^{(2\gamma \log(1/\gamma)+\gamma)\sqrt{n}-\log(en)}\geq e^{(2\gamma \log(1/\gamma)+\gamma/2)\sqrt{n}}\;,
\end{align*}
where the second inequality holds provided $n_2>1/\gamma ^2$, and the third inequality holds provided $\log(en_2)\leq \gamma \sqrt{n_2}/2$.
\end{proof}

The second key lemma we will need is Lemma \ref{Theorem - France} below.
We remark that \eqref{Thm - a result of [2]} below is Theorem 2.6 in \cite{DN}, see also \cite{Mosaki} for a refined version of this result.
We give a self contained elementary proof of \eqref{Thm - a result of [2]} which is significantly simpler and also allows us to derive the stronger statement stated after \eqref{Thm - a result of [2]}.

\begin{lemma}\label{Theorem - France}
There exists a positive real $\lambda_0$ such that for every $\lambda \geq \lambda_0$ there exists $n_3=n_3(\lambda)$ such that for every integer $n>n_3$ we have
\begin{equation} \label{Thm - a result of [2]}
    \log p_{[\lambda \sqrt{n},n]}(n)=\left(\frac{2\log(\lambda) \pm \Theta(\log\log(\lambda))}{\lambda}\right)\sqrt{n}\;.
\end{equation}
Furthermore, for every positive real $\varepsilon\leq 1$ the lower bound holds also for $\log p_{[\lambda\sqrt{n},\varepsilon n]}(n)$ provided $\lambda \geq \lambda_0(\varepsilon)$ and $n\geq n_3(\lambda,\varepsilon)$.
\end{lemma}

\begin{proof}
Let $\lambda\geq \lambda_0$ be a real number where the value of $\lambda_0$ will be specified later and let $n$ be an integer with $n>n_3=n_3(\lambda)$ which will be specified later.
We first claim that
\begin{equation*}
    p_{[\lambda \sqrt{n},n]}(n)\leq p_{ [\sqrt{n}/\lambda] }(n)=p_{\sqrt{n}/\lambda}(n+\sqrt{n}/\lambda)\;.
\end{equation*}
To justify the inequality we observe that each partition of $n$ with all integers from $[\lambda \sqrt{n},n]$ uses at most $\sqrt{n}/\lambda$ integers.
As noted earlier it is well know that the number of partitions of $n$ with at most $\sqrt{n}/\lambda$ parts is precisely $p_{[\sqrt{n}/\lambda]}(n)$. Finally, the equality holds by Lemma \ref{Lemma - at most k parts}.
Applying Lemma \ref{Lemma - estimation for p_k} invoked with $n$ replaced by $n+\sqrt{n}/\lambda$ and $k=\sqrt{n}/\lambda $ we obtain,
\begin{align*}
   p_{[\lambda \sqrt{n},n]}(n)&\leq \frac{\binom{n+\sqrt{n}/\lambda +\binom{\sqrt{n}/\lambda}{2}-1}{\sqrt{n}/\lambda -1}}{(\sqrt{n}/\lambda)!}\leq  \frac{\binom{3n}{\sqrt{n}/\lambda}}{\left(\sqrt{n}/\lambda \right)!}\\
   &\leq \left ({3e\lambda\sqrt{n}}\right)^{\sqrt{n}/\lambda}\left(\frac{e\lambda }{\sqrt{n}}\right)^{\sqrt{n}/\lambda}
   = \left(3e^2\lambda^2\right)^{\sqrt{n}/\lambda}\leq e^{\left(\frac{2\log(\lambda)+4}{\lambda}\right)\sqrt{n}}\;,
\end{align*}
where the second inequality holds provided $\sqrt{n_3}/\lambda \leq n_3$, $\sqrt{n_3}/\lambda -1\leq 3 n_3/2$ and provided $\lambda_0>1$ and as $\binom{\sqrt{n}/\lambda}{2}-1\leq n/\lambda^2$.
This concludes the proof of the upper bound of \eqref{Thm - a result of [2]}.

For the lower bound of \eqref{Thm - a result of [2]} we claim that for any positive integer $k$ we have
\begin{equation}\label{equation - integers}
    p_{[\lambda \sqrt{n},n]}(n)\geq p_{k}(n-k\cdot \lfloor\lambda\sqrt{n}\rfloor)\;.
\end{equation}
To see this we define the following one to one correspondence. For every partition of $n-k\cdot \lfloor\lambda\sqrt{n}\rfloor$ with positive integers $(x_i)_{i=1}^{k}$ we define $(y_i)_{i=1}^{k}$ with $y_i=x_i+\lfloor\lambda \sqrt{n}\rfloor$. This is clearly a one to one correspondence and furthermore $(y_i)_{i=1}^{k}$ is partition of $n$ with all parts taken from $[\lambda \sqrt{n},n]$. Setting

\begin{equation*}\label{equation - definition of k}
    k=\left\lfloor\frac{\sqrt{n}}{\lambda+\lambda/\log(\lambda) }\right\rfloor
\end{equation*}
in \eqref{equation - integers} and applying Lemma \ref{Lemma - estimation for p_k} we obtain,
\begin{align*}
    p_{[\lambda \sqrt{n},n]}(n)&\geq \frac{\binom{n-k\lfloor\lambda n\rfloor-1}{k-1}}{k!}\\
    &\geq {\binom{\frac{n}{\log(\lambda)+1}-1}{\frac{\sqrt{n}}{\lambda+\lambda/\log(\lambda)}-1}}\Bigg/{\left(\frac{\sqrt{n}}{\lambda+\lambda/\log(\lambda)}\right)!} \\
    &\geq {\binom{\frac{n}{2(\log(\lambda)+1)}}{\frac{\sqrt{n}}{\lambda+\lambda/\log(\lambda)}-1}}\Bigg/{\left(\frac{\sqrt{n}}{\lambda+\lambda/\log(\lambda)}\right)!}\\
    &\geq \left(\frac{\lambda \sqrt{n}}{2\log(\lambda)}\right)^{\frac{\sqrt{n}}{\lambda+\lambda/\log(\lambda)}}\left(\frac{e(\lambda+\lambda/\log(\lambda))}{\sqrt{n}}\right)^{\frac{\sqrt{n}}{\lambda+\lambda/\log(\lambda)}}\frac{2(\log(\lambda)+1)}{e{n}}\\
    &=\left(e{\lambda ^2\left(\frac{1}{2\log(\lambda)}+\frac{1}{2\log^2(\lambda)}\right)}\right)^{\frac{\sqrt{n}}{\lambda}-\frac{\sqrt{n}}{\lambda(\log(\lambda)+1)}}\frac{2(\log(\lambda)+1)}{e{n}}\\
    &\geq \left(\frac{\lambda }{\log(\lambda)}\right)^{\frac{2\sqrt{n}}{\lambda}-\frac{4\sqrt{n}}{\lambda\log(\lambda)}}=\exp\left(\left(\log(\lambda)-\log\log(\lambda)\right)\left(\frac{2\sqrt{n}}{\lambda}-\frac{4\sqrt{n}}{\lambda\log(\lambda)}\right)\right)\\
    &\geq \exp\left(\frac{\left(2\log(\lambda)-3\log\log(\lambda)\right)\sqrt{n}}{\lambda}\right)\;,
\end{align*}
where the third inequality holds provided $\frac{n_3}{2(\log(\lambda)+1)}\geq 1$ and provided $\lambda_0\geq 1$, the fifth inequality holds provided $\log\left(\frac{en_3}{2(\log(\lambda)+1)}\right)\leq (\log(\lambda)-\log\log(\lambda))\frac{\sqrt{n_3}}{\lambda \log(\lambda)}$, and the sixth inequality holds provided $\log\log(\lambda_0)\geq 4$.

As for the furthermore part of the theorem, fix $0< \varepsilon\leq 1$ and note the correspondence we used above when proving the lower bound of \eqref{Thm - a result of [2]},
took partitions of $n-k\cdot\lfloor\lambda\sqrt{n}\rfloor$ that use $k$ integers and mapped them to partitions of $n$ using integers in $[\lambda \sqrt{n},\left(2\lambda+\frac{\log(\lambda)}{\lambda+\lambda\log(\lambda)}\right)\sqrt{n}+\frac{n}{\log(\lambda)+1}]$.
Therefore, if we assume that $\frac{2}{\varepsilon}\leq \log(\lambda_0)+1$ and $\sqrt{n_3}\geq \frac{2}{\varepsilon}\left(2\lambda+\frac{\log(\lambda)}{\lambda+\lambda\log(\lambda)}\right) $ then
we in fact obtain the same lower bound stated in \eqref{Thm - a result of [2]} even if using only integers from the interval
$[\lambda\sqrt{n},\varepsilon n]$.
\end{proof}

We now use Lemmas \ref{Lemma - main lemma liminf lower bound}, \ref{Theorem - szekeres} and \ref{Theorem - France} in order to prove Theorem \ref{theolower} equation \eqref{eqTheo3}.

\begin{proof}[Proof of Theorem \ref{theolower} equation \eqref{eqTheo3}]
By \eqref{equation - Hardy Ramanujan} it is sufficient to prove that there exists a set of positive integers $A$ with lower density $\alpha$ and $\gcd(A)=1$ satisfying
\begin{equation}\label{equation - first equation upper bound}
        \log p_{A}(m)\geq (2\log(1/\alpha)-\Theta(\log\log(\alpha)))\sqrt{\alpha m}\;.
\end{equation}
To this end suppose $\alpha<\alpha_0$ where $\alpha_0$ is a small positive real which will be specified later. Let $n$ be an integer with $n>n_0$ where $n_0=n_0(\alpha)$ is some positive integer which will be specified later and will also be greater than $1/\alpha$.
We claim that the set $A$ which we introduce next satisfies \eqref{equation - first equation upper bound}. Define a sequence of sets $A_i$ recursively as follows. Set $f(0)=1$, and $f(1)=n_0$, and for any positive integer $i$ let $f(i+1)={f(i)^{2}}$, $A_{i+1}=[f(i), \alpha f(i+1)]$. Finally take $A=\bigcup_{n\geq 1}A_{n}$. It is easy to see that the lower density of $A$ is $\alpha$, and since $n_0>1/\alpha$ we have $1\in A_1\subseteq A$ which implies that $\gcd(A)=1$.

Provided
\begin{equation}\label{equation - condition n_0 1}
    n_0>n_1(\alpha)
\end{equation}
we may use Lemma \ref{Lemma - main lemma liminf lower bound} invoked with $n$ replaced by $f(i)\geq n_0$ for $i\geq 1$, which asserts the following for all $m\in [16 \alpha f(i+1),\alpha f(i+2)/16]$,
\[
    p_{\{1\}\cup[f(i),\alpha f(i+1)]}(m) \geq e^{(2\log(1/\alpha)-8)\sqrt{\alpha m}}\;.
\]
For all $i\geq 2$ the set $\{1\}\cup [f(i),\alpha f(i+1)]$ is a subset of $A$ and thus we obtain \eqref{equation - first equation upper bound} for all $m\geq \alpha f(3)/16$ except for $m\in [\alpha f(i+1)/16,16\alpha f(i+1)]$ where $i\geq 2$.
Therefore, to complete the proof of \eqref{equation - first equation upper bound} it remains to consider only $m\in [\alpha f(i+1)/16,16 \alpha f(i+1)]$ where $i\geq 2$ is some integer.

Therefore for the rest of the proof let us fix $i\geq 2$. For simplicity of presentation denote $f(i+1)$ by $n^2$ and then $f(i)={n}$ and $f(i-1)=\sqrt{n}$. Let $c$ be a real number with $1/16\leq c\leq 16$ and set
\[
    m=c\cdot \alpha\cdot n^2\;.
\]
Since $A$ contains both $\{1\}\cup A_{i}=\{1\}\cup[\sqrt{n},\alpha n]$ and $A_{i+1}=[n,\alpha n^2]$ we deduce the following for all $0\leq \delta \leq 1$,
\begin{align*}
    p_{A}(m)&\geq \sum_{k=0}^{m}p_{\{1\}\cup [\sqrt{n},\alpha n]}(k)\cdot p_{[n,\alpha n^2]}(m-k)\\
    &\geq p_{\{1\}\cup [\sqrt{n},\alpha n]}(\delta m)\cdot p_{[n,\alpha n^2]}((1-\delta) m)\;.
\end{align*}
Hence to establish \eqref{equation - first equation upper bound} it is enough to show that there exists $0<\delta < 1$ such that for all $1/16\leq c\leq 16$ we have
\begin{equation}\label{what we need 1}
    \log(p_{\{1\}\cup [\sqrt{n},\alpha n]}(\delta m))+\log(p_{[n,\alpha n^2]}((1-\delta) m))\geq (2\log(1/\alpha)-\Theta(\log\log(1/\alpha)))\sqrt{\alpha m}\;.
\end{equation}
To simplify \eqref{what we need 1} we observe that
\begin{equation} \label{equation - Szekeres first estimation}
     p_{[\alpha n]}(\delta m)= \sum _{k=0}^{\delta m} p_{[2,\sqrt{n}-1]}(k)\cdot p_{\{1\}\cup [\sqrt{n},\alpha n]}(\delta m-k)\leq(\delta m+1)(\delta m+1)^{\sqrt{n}}\cdot p_{\{1\}\cup [\sqrt{n},\alpha n]}(\delta m) \;,
\end{equation}
where the inequality holds by Claim \ref{Claim - trivial bound} and by the monotonicity of $p_{\{1\}\cup [\sqrt{n},\alpha n]}$.
Hence, provided
\begin{equation}\label{equation - condition n_0 -1}
    \alpha \log\log(1/\alpha)\sqrt{n_0}/16\geq 2\log(8n_0)\;,
\end{equation}
to obtain \eqref{what we need 1} it is enough to prove that there exists $0<\delta <1$ such that the following holds for all $1/16\leq c \leq 16$
\begin{equation}\label{equation - what we need 2}
    \log(p_{[\alpha n]}(\delta m))+\log(p_{[n,\alpha n^2]}((1-\delta)m))\geq (2\log(1/\alpha)-\Theta(\log\log(1/\alpha)))\sqrt{\alpha m}\;.
\end{equation}

Provided $\delta$ and $\alpha$ satisfy the right-hand inequality (the left-hand inequality holds as $c\geq 1/16$)
\begin{align}\label{equation - condition alpha 1.5}
    \frac{\alpha}{\delta \cdot c}\leq \frac{\alpha}{\delta/16}\leq 1\;,
\end{align}
and $n_0$ satisfies
\begin{equation}\label{equation - condition n_0 2}
    \delta m \geq \delta \alpha\cdot n_0/16 > n_2\left(\sqrt{\frac{\alpha}{16\delta}}\right)\geq n_2\left(\sqrt{\frac{\alpha}{\delta\cdot c}}\right)\;,
\end{equation}
we may apply Lemma \ref{Theorem - szekeres} invoked with $\gamma$ replaced by $\sqrt{\frac{\alpha}{\delta \cdot c}}$ and with $n$ replaced by $\delta m$ and obtain,
\begin{align}
    \log\left(p_{[\alpha n]}(\delta m)\right) &\geq \left({2\sqrt{\frac{\alpha}{\delta \cdot c}}}\log\left(\sqrt{{\frac{c\cdot \delta}{\alpha }}}\right)+\Theta\left(\sqrt{\frac{\alpha}{\delta \cdot c}} \right)\right)\sqrt{\delta m} \nonumber\\
    &= \left({\alpha }\log\left({\frac{c\cdot \delta}{\alpha }}\right)+\Theta\left(\alpha \right)\right)n \nonumber\\
    &\geq \left({\alpha }\log\left({\frac{\delta}{\alpha }}\right)+\Theta\left(\alpha \right)\right)n \label{equation - first part sharp estimation}\;,
\end{align}
where the second inequality holds as $c\geq 1/16$.

From now on let us assume (with foresight) that the $\delta$ which establishes \eqref{equation - what we need 2} is less than $1/2$.
Provided $\alpha_0$ is small enough so that
\begin{equation}\label{equation - condition alpha 1}
    {\frac{1}{\sqrt{(1-\delta)c\cdot\alpha}}} \geq {\frac{1}{\sqrt{16\alpha}}}\geq \lambda_0\left(\frac{1}{16}\right)\geq \lambda_0\left(\frac{1}{16(1-\delta)}\right)\;,
\end{equation}
and $n_0$ satisfies
\begin{equation}\label{equation - condition n_0 3}
    (1-\delta)m\geq \frac{\alpha \cdot n_0^2}{16}\geq n_3\left({\frac{1}{\sqrt{\alpha/32}}},\frac{1}{16}\right)\geq n_3\left(\frac{1}{\sqrt{(1-\delta)c\cdot \alpha}},\frac{1}{16(1-\delta)}\right)\;,
\end{equation}
we may apply Lemma \ref{Theorem - France} invoked with $\varepsilon=\frac{1}{16(1-\delta)}$, $\lambda=\frac{1}{\sqrt{(1-\delta)c\cdot \alpha}}$,
and with $n$ replaced by $(1-\delta)m$. As $1/16\leq c\leq 16$ and $0<\delta<1/2$ we obtain,
\begin{align}\label{equation - with c}
    \log\left(p_{[n,m/16]}((1-\delta)m)\right)&\geq \left(\frac{2\log\left(1/{\sqrt{(1-\delta)c\cdot \alpha}}\right)-\Theta\left(\log\log\left(1/{\sqrt{(1-\delta)c\cdot \alpha}}\right)\right)}{1/{\sqrt{(1-\delta)c\cdot \alpha}}}\right)\sqrt{(1-\delta)m}\nonumber\\
    &\geq (1-\delta)c \cdot \alpha \left(\log\left(\frac{1}{\alpha}\right)-\Theta\left(\log\log\left(\frac{1}{{\alpha}}\right)\right)\right)n \;.
\end{align}
Since $m/16\leq \alpha n^2$ \eqref{equation - with c} implies
\begin{align}\label{equation - second part sharp estimation}
    \log\left(p_{[n,\alpha n^2]}((1-\delta)m)\right)\geq (1-\delta)c\cdot\alpha \log\left(\frac{1}{\alpha}\right)n-\Theta\left(\alpha\log\log\left(\frac{1}{{ \alpha}}\right)\right)n\;.
\end{align}

All that is left is to choose the optimal $\delta$ that will maximize the sum of \eqref{equation - first part sharp estimation} and \eqref{equation - second part sharp estimation}.
It is not hard to see that (up to lower order terms) the optimal choice is $\delta=1/\log(1/\alpha)$, and that
with this choice of $\delta$, we can choose $\alpha_0$ small enough so that \eqref{equation - condition alpha 1.5} and \eqref{equation - condition alpha 1} will hold,
and then choose $n_0$ large enough so that $n_0 > 1/\alpha$ and \eqref{equation - condition n_0 1}, \eqref{equation - condition n_0 -1}, \eqref{equation - condition n_0 2} and \eqref{equation - condition n_0 3} will hold. Plugging $\delta = 1/\log(1/\alpha)$ in \eqref{equation - first part sharp estimation} we obtain,
\begin{align}\label{equation - last 1}
    \log\left(p_{[\alpha n]}(\delta m)\right) &\geq \left({\alpha }\log\left({\frac{1}{\alpha\log(1/\alpha) }}\right)+\Theta\left(\alpha \right)\right)n=\alpha\log(1/\alpha)n-\Theta(\alpha \log\log(1/\alpha))n\;.
\end{align}
Similarly plugging $\delta = 1/\log(1/\alpha)$ in \eqref{equation - second part sharp estimation} we obtain,
\begin{align}
    \log\left(p_{[n,\alpha n^2]}((1-\delta)m)\right)&\geq \left(1-\frac{1}{\log(\alpha)}\right)c\cdot\alpha \log\left(\frac{1}{\alpha}\right)n-\Theta\left(\alpha\log\log\left(\frac{1}{{\alpha}}\right)\right)n\nonumber\\
    &\geq c\cdot \alpha\log(1/\alpha)n-\Theta(\alpha \log\log(1/\alpha))n\label{equation - last 2}
\end{align}
Combining \eqref{equation - last 1} and \eqref{equation - last 2} we obtain
\begin{align*}
        \log(p_{[\alpha n]}(\delta m))+\log(p_{[n,\alpha n^2]}((1-\delta)m))&\geq ((1+c)\log(1/\alpha)-\Theta(\log\log(1/\alpha)))\alpha n\\
        &= ((1/\sqrt{c}+\sqrt{c})\log(1/\alpha)-\Theta( \log\log(1/\alpha)))\sqrt{\alpha m}\\
        &\geq (2\log(1/\alpha)-\Theta(\log\log(1/\alpha))) \sqrt{\alpha m}\;,
\end{align*}
where the last inequality holds as $1/x+x\geq 2$ for all $x>0$. This is \eqref{equation - what we need 2}, and the proof is completed.
\end{proof}

We now turn to the proof of Theorem \ref{theolower} equation \eqref{eqTheo4}. We will need the following lemma.

\begin{lemma}\label{Lemma - main lemma upper bound on liminf}
There exists $\alpha_0>0$ such that for all $0<\alpha<\alpha_0$ there exists an integer $n_0=n_0(\alpha)$ such that the following holds for every integer $n>n_0$. If $A$ is a set of positive integers with $|A\cap [n]|=\alpha n$ then we have the following where $m=\alpha n^2$,
\[
    p_{A}(m)\leq e^{(2\log(1/\alpha) +\Theta(\log\log(1/\alpha)))\sqrt{\alpha m}}\;.
\]
\end{lemma}

\begin{proof}
Let $A_1=A\cap [n]$ and $A_2=A\cap [n+1,m]$. We have,
\begin{align*}
    p_A(m)=\sum_{k=0}^{m} p_{A_1}(k)\cdot p_{A_2}(m-k)\leq (m+1)\cdot \max_{k\in [m]} p_{A_1}(k)\cdot \max_{k\in [m]} p_{A_2}(k)\;.
\end{align*}
Hence it is enough to show that the following holds for all $k\in [m]$ and $j=1,2$,
\begin{equation}\label{equation - last equation}
    p_{A_j}(k)\leq e^{(\log(1/\alpha) +\Theta(\log\log(1/\alpha)))\sqrt{\alpha m}}\;.
\end{equation}
Using Lemma \ref{Lemma - the first part is always the best} and the monotonicity of $p_{[\alpha n]}$ we deduce that,
\begin{align*}
    p_{A_1}(k)&\leq p_{[\alpha n]}(k)\leq p_{[\alpha n]}(m)=p_{[\sqrt{\alpha m}]}(m)\;.
\end{align*}
Provided $n_0$ is large enough we may use Lemma \ref{Theorem - szekeres} invoked with $n$ replaced by $m$ and $\gamma $ replaced by $\sqrt{\alpha}$ and obtain
\[
    \log p_{[\sqrt{\alpha m}]}(m)\leq (2\sqrt{\alpha} \log(1/\sqrt{\alpha})+ \Theta(\sqrt{\alpha}))\sqrt{m}= (\log(1/{\alpha})+ \Theta(1))\sqrt{\alpha m}\;.
\]
Combining the above we obtain \eqref{equation - last equation} for $j=1$.

We now prove \eqref{equation - last equation} for $j=2$. We first note that if $k\in [n]$ then \eqref{equation - last equation} trivially holds since in this case $p_{A_2}(k)=0$. Assume now $k\in[n+1,m]$. We may use Lemma \ref{Theorem - France} with $\lambda=1/\sqrt{\alpha}$ and $n$ replaced by $k$ provided $\alpha_0$ is small enough so that $1/\sqrt{\alpha_0}>\lambda_0$ and $k>n_0>n_3(1/\sqrt{\alpha})$. We obtain
\begin{align*}
    \log p_{[\sqrt{k/\alpha},k]}(k)&\leq \left(\frac{2\log(1/\sqrt{\alpha}) + \Theta(\log\log(1/\sqrt{\alpha}))}{1/\sqrt{\alpha}}\right)\sqrt{k}\\
    &= \left({\log(1/\alpha) +\Theta(\log\log(1/\alpha))}\right)\sqrt{\alpha k}\\
    &\leq \left({\log(1/\alpha) +\Theta(\log\log(1/\alpha))}\right)\sqrt{\alpha m}.
\end{align*}
Since $k\leq m=\alpha n^2$ then $\sqrt{k/\alpha}\leq n$ and therefore $A_2\cap [k]\subseteq [n+1,k]\subseteq[\sqrt{k/\alpha},k]$ and thus
\[
    p_{A_2}(k)=p_{A_2\cap [k]}(k)\leq p_{[\sqrt{k/\alpha},k]}(k).
\]
Combining the above we obtain \eqref{equation - last equation} for $j=2$.
\end{proof}

Now using the above lemma we prove Theorem \ref{theolower} equation \eqref{eqTheo4}.

\begin{proof}[Proof of Theorem \ref{theolower} equation \eqref{eqTheo4}]
Suppose $\alpha<\alpha_0$ where $\alpha_0$ is given by Lemma \ref{Lemma - main lemma upper bound on liminf} and $A$ is a set of positive integers with lower density $\alpha$. Similar to the proof of Theorem \ref{theolower} equation \eqref{eqTheo3} it is sufficient to prove that there exist infinitely many pairs of integer and real number $(m_i,\alpha_i)$ satisfying
\begin{equation}\label{equation - truly last}
    \log p_{A}(m_i)\leq {(2\log(1/\alpha_i)+\Theta(\log\log(1/\alpha_i)))\sqrt{\alpha_i m_i}},\;\lim_{i\to \infty}\alpha_i=\alpha, \mbox{ and } \lim_{i\to \infty}m_i=\infty\;.
\end{equation}
Since $A$ has lower density $\alpha$ there exists an increasing sequence $\{n_i\}_{i=0}^\infty$ of positive integers such that setting $A_{i}=A\cap [n_i]$ and $\alpha_i=|A_i|/n_i$ we have $\lim_{i \to \infty}\alpha _i=\alpha$. Fix $i_0$ large enough so that for all $i>i_0$ we have $\alpha/2<\alpha_i<\alpha_0$ and $n_i>n_0(\alpha/2)$ where $n_0(\alpha/2)$ is given by Lemma \ref{Lemma - main lemma upper bound on liminf}. Now we may use Lemma \ref{Lemma - main lemma upper bound on liminf} invoked with $\alpha$ replaced by $\alpha_i$ and $n$ replaced by $n_i$ and obtain
\[
    p_{A}(m_i)\leq e^{{2(\log(1/\alpha_i) +\Theta(\log\log(1/\alpha_i)))\sqrt{\alpha_i m_i}}}\;,
\]
where $m_i=\alpha_i n_i^2$. Since $\lim_{i\to \infty}\alpha_i=\alpha$ and $\lim_{i\to \infty}m_i=\infty$ we obtain \eqref{equation - truly last}, thus completing the proof.
\end{proof}

\section{Proof of Theorem \ref{theoupper}}\label{secproof}

We start with a proof of Theorem \ref{theoupper} equation \eqref{eqTheo2} and then move on to proving Theorem \ref{theoupper} equation \eqref{eqTheo1}.

\begin{lemma}\label{Lemma - pigeonhole principle}
Suppose $m$ is a positive integer and $A\subseteq [m]$ with $\beta=|A|/m$. Then, there exists $n \in [\beta^2 m^2/4,\beta(1-\beta/4)m^2]$ satisfying
\[
    \log p_{A}(n)\geq (2\log(2)-o_n(1))\sqrt{\frac{\beta n}{1-\beta/4}}\;.
\]
\end{lemma}

\begin{proof}
Let $A_1$ be the set of $\beta m/2$ smallest integers in $A$ and $A_2$ the set of $\beta m/2$ largest integers in $A$. Note that any $\beta m/2$ integers in $A_1$ sum up to an integer in $[\beta(1-\beta/2) m^2/2]$. Furthermore, the number of options to choose $\beta m/2 $ integers from $A_1$ (not necessarily distinct) is exactly the same as the number of nonnegative solutions to $\sum_{i=1}^{\beta m/2} x_i=\beta m/2$. That can be seen by taking $x_i$ to be the number of times the $i^{th}$ smallest element in $A_1$ is taken. Hence there are $\binom{\beta m-1}{\beta m/2}$ such choices.
Similarly, any $\beta m/2$ integers in $A_2$ (not necessarily distinct) sum up to an integer in $[\beta^2 m^2/4,\beta m^2/2]$. Similar to before, the number of ways to choose these integers is $\binom{\beta m-1}{\beta m/2}$. Therefore, the number of ways one can choose $\beta m$ integers from $A$, such that half of them are taken from $A_1$ and the other half from $A_2$ is $\binom{\beta m-1}{\beta m/2}^2$. Since the sum of these integers is an integer in $[\beta^2 m^2/4,\beta(1-\beta/4)m^2]$ we infer by the pigeonhole principle that there exists $n\in[\beta ^2m^2/4,\beta(1-\beta/4)m^2]$ satisfying
\[
    p_{A}(n)\geq \frac{\binom{\beta m-1}{\beta m/2}^2}{m^2}= \frac{\binom{\beta m}{\beta m/2}^2}{4m^2}\geq \frac{2^{2\beta m}}{4(\beta m+1)^2m^2}= 2^{(2-o_n(1))\beta m}\geq 2^{(2-o_n(1))\sqrt{\frac{\beta n}{1-\beta/4} }}\;.\qedhere
\]
\end{proof}

\begin{proof}[Proof of Theorem \ref{theoupper} equation \eqref{eqTheo2}]
Let $A$ be a set of positive integers with upper density $\beta<1/2$. By \eqref{equation - Hardy Ramanujan} it is sufficient to prove that there exist infinitely many pairs of integer and real numbers $(n_i,\beta_i)$ satisfying
\begin{equation}\label{equation - what we need in 1.2}
     \log p_{A}(n_i)\geq (2\log(2)-o_{n_i}(1))(1+\beta_i/10)\sqrt{{\beta_i n_i}}\;,
\end{equation}
$\lim_{i\to \infty} n_i =\infty$ and $\lim_{i\to \infty} \beta_i=\beta$.
Since $A$ has upper density $\beta$ there exists an increasing sequence of integers $\{m_i\}_{i=0}^{\infty}$ such that setting $A_i=A\cap[m_i]$ and $\beta_i = {|A_i|}/{m_i}$ we have $\lim_{i\to \infty} \beta_{i}=\beta$. Lemma \ref{Lemma - pigeonhole principle} implies that for every $i$ there exists $n_i\in [\beta_i^2 m_i^2/4,\beta_i(1-\beta_i/4)m_i^2]$ satisfying
\[
    \log p_{A}(n_i) \geq {(2\log(2)-o_{n_i}(1))\sqrt{\frac{\beta_i n_i}{1-\beta_i/4}}}\geq {(2\log(2)-o_{n_i}(1))(1+\beta_i/10)\sqrt{\beta_i n_i}}\;.
\]
Since $\lim_{i\to \infty} n_i =\infty$ and $\lim_{i\to \infty} \beta_i=\beta$ we obtain \eqref{equation - what we need in 1.2}.
\end{proof}

To prove Theorem \ref{theoupper} equation \eqref{eqTheo1} require the following lemma.

\begin{lemma}\label{Main lemma}
Suppose $0<\beta<1$. Then for all $n,m$ positive integers we have,
\[
    p_{[(1-\beta)n+1,n]}(m)\leq e^{2\log(2)\sqrt{\frac{\beta }{1-\beta} m}}\;.
\]
\end{lemma}

\begin{proof}[Proof of Lemma \ref{Main lemma}]
Assume $n>1/\beta$\footnote{ This can be assumed as otherwise $[(1-\beta)n+1,n]$ is empty and the lemma holds trivially.} and set $\gamma=m/n^2$. For simplicity of presentation we let $A$ denote the set $[(1-\beta)n+1,n]$. Our goal is then to prove that for any $\gamma>0$ we have,
\begin{equation}\label{equation - what we need limsup}
    p_{A}(\gamma n^2)\leq 2^{2\sqrt{\frac{\beta \gamma}{1-\beta}}n}\;.
\end{equation}
Observe that a partition of $\gamma n^2$ using integers from $A$ can use at most $\gamma n/(1-\beta)$ numbers. Hence, we can encode each such partition as a solution in nonnegative integers to the inequality $\sum _{i=0}^{\beta n-1} x_i \leq  \gamma n/(1-\beta)$. This is done by taking $x_i$ to be the number of times $n-i$ appears in the partition. Therefore we obtain,
\begin{align}\label{equation - bound on p_A(bn)}
    p_{A}(\gamma n^2 )&\leq \binom{\gamma n/(1-\beta)+\beta n}{\gamma n/(1-\beta)}\;.
\end{align}
We now use the well known inequality,
\begin{equation}\label{equation - entropy inequality}
    \binom{n}{k}\leq 2^{H_2(k/n)n}\;,
\end{equation}
where $H_2(x)$ is the binary entropy function defined by
\[
    H_2(x)=-x\log_2(x)-(1-x)\log_2(1-x)\;.
\]
From \eqref{equation - bound on p_A(bn)} and \eqref{equation - entropy inequality} we obtain,
\begin{equation}\label{equation - a mid bound on p_A(bn)}
    p_{A}(\gamma n^2 )\leq 2^{H_2\left(\frac{\gamma}{\gamma+\beta(1-\beta)}\right)(\gamma /(1-\beta)+\beta )n}\;.
\end{equation}
Let
\[
    f_\beta(\gamma)=H_2\left(\frac{\gamma}{\gamma+\beta(1-\beta)}\right)\frac{\gamma /(1-\beta)+\beta }{\sqrt{\gamma }}\;,
\]
and observe that to prove \eqref{equation - what we need limsup} it is enough to prove that $f_\beta(\gamma)\leq 2\sqrt{\frac{\beta}{{1-\beta}}}$. Noting that $f_{\beta}(\beta(1-\beta))=2\sqrt{\frac{\beta}{{1-\beta}}}$ it is enough to prove that $f_\beta'(\gamma)>0$ for $0<\gamma < \beta(1-\beta)$ and $f_\beta'(\gamma)<0$ for $\gamma>\beta(1-\beta)$.

We first note that
\[
  f_\beta'(\gamma)=\frac{ \beta\left(1 - \beta\right)  \log_2\left(\frac{\beta(1-\beta)}{\beta\left(1 - \beta\right) + \gamma}\right)-\gamma \log_2\left(\frac{\gamma}{\beta\left(1-\beta\right) + \gamma}\right)}{2 \left(1-\beta\right) \gamma^{3/2}}\;.
\]
Since the denominator above is always positive, we focus on the nominator. For every $a > 0$ define
\[
    g_a(x)=
     a\log_2\left(\frac{a}{a+x}\right)-x\log_2\left(\frac{x }{a+x}\right)\;.
\]
We will now show that if $0<x<a$ then $g_a(x)>0$ and if $x>a$ then $g_a(x)<0$, noting that this implies the required assertion regarding $f_\beta'(\gamma)$ upon taking $a = \beta(1-\beta)$.
Differentiating $g_a(x)$ we obtain
\[
   g_a'(x)=\frac{-2a - (a + x) \log\left(\frac{x}{a + x}\right)}{(a + x) \log(2)}\;.
\]
To see that $g_a(x)<0$ for $x>a$ observe that $-(a+x)\log\left(\frac{x}{a+x}\right)$ is a decreasing function as its derivative is $-\frac{a}{x}+\log\left(1+\frac{a}{x}\right)$ which is negative by the well known inequality $\log(1+x)<x$.
Therefore $-(a+x)\log\left(\frac{x}{a+x}\right)\leq -2a\log(1/2)$ implying
\[
    g_a'(x)\leq \frac{-2a(1+\log(1/2))}{(a+x)\log(2)}< 0\;.
\]
Thus, $g_a(x)$ is strictly decreasing for all $x>a$, implying that $g_a(x)<0$ for all $x>a$ as $g_a(a)=0$.

We now prove that $g_a(x)> 0$ for all $0<x< a$. To this end note that
\[
    g_a''(x)=\frac{-a (a - x)}{x (a + x)^2 \log(2)}\;,
\]
which is negative for all $0<x<a$. Therefore $g_a(x)$ is concave in $(0,a)$, implying that $g_a(x)>0$ for all $x \in (0,a)$ as $g_a(a)=0$ and $\lim_{x \to 0}g_a(x)=0$. Thus we have completed the proof that if $0<x<a$ then $g_a(x)>0$ and if $x>a$ then $g_a(x)<0$.
\end{proof}

\begin{proof}[Proof of Theorem \ref{theoupper} equation \eqref{eqTheo1}]
Similar to the proof of Theorem \ref{theoupper} equation \eqref{eqTheo2} it is enough to prove that there exists a set of positive integers $A$ with upper density $\beta$ and $\gcd (A)=1$ satisfying,
\begin{equation}\label{equation - what we need thm 1.2 part 2}
    \log p_{A}(m)\leq {(2\log(2)+o_m(1))\sqrt{\frac{\beta m }{1-\beta}}}\leq (2\log(2)+o_m(1))(1+\beta)\sqrt{{\beta m }}\;,
\end{equation}
where the second inequality holds provided $\beta \leq 1/2$.
To this end fix $n_0=\lceil\frac{1}{\beta}\rceil$. We define $A$ as follows. Define a sequence of sets $A_i$ recursively. Set $A_0=[(1-\beta) n_0+1,n_0]$ and $f(0)=n_0$, for any positive integer $i$ let $f(i+1)=2^{f(i)}$ and $A_{i+1}=[(1-\beta) f(i+1)+1,{f(i+1)}]$. Now let $A=\bigcup_{n\in \mathbb{N}}A_{n}\;.$
It is easy to see that the upper density of $A$ is $\beta$. Further, since $A$ contains two consecutive integers we have $\gcd(A)=1$. We now prove that this set satisfies the first inequality in \eqref{equation - what we need thm 1.2 part 2}.
Let $m$ be any positive integer greater than $n_0$ and let $i$ be the unique integer such that $m\in[f(i)+1,f(i+1)]$. For simplicity of presentation we set $n=f(i)$ and thus $f(i+1)=2^n$. We consider two cases, the first is when $m\in[n+1,(1-\beta) 2^n]$ and the second is when $m\in [(1-\beta) 2^n+1,2^n]$.
\begin{casesp}
\item
Assume $m\in[n+1,(1-\beta) 2^n]$. Let $B=\bigcup_{j \leq i-1}A_{j}$ and note that we have,
\[
    p_A(m) = \sum_{0 \leq k \leq m} p_{B}(k) \cdot p_{A_{i}}(m-k)\;.
\]
Note further that our choice of $f$ guarantees that $|B|\leq \log_{2}(n)\leq \log_{2}(m)$. This, Claim \ref{Claim - trivial bound} and Lemma \ref{Main lemma} give the following bound
\begin{align*}
    p_{A}(m)&\leq \sum_{0 \leq k \leq m} (k+1)^{\log_{2}({m})} \cdot e^{2\log(2){\sqrt{\frac{\beta(m-k)}{1-\beta}}}}\\
    &\leq (m+1)(m+1)^{\log_{2}(m)}e^{2\log(2)\sqrt{\frac{\beta m}{1-\beta} }}\\
    &\leq e^{(2\log(2)+o_m(1))\sqrt{\frac{\beta m}{1-\beta}}}\;.
\end{align*}
Taking logarithm from both sides of the inequality we obtain the first inequality in \eqref{equation - what we need thm 1.2 part 2}.
\item
Assume $m\in [(1-\beta) 2^n+1,2^n]$. Let $B=\bigcup_{j\leq  i}A_{j}$ and note that we have,
\begin{align*}
     p_A(m) &= \sum_{0 \leq k \leq m} p_{B}(m-k) \cdot p_{A_{i+1}}(k)\\
                           &= \sum_{0 \leq k \leq (1-\beta) 2^n} p_{B}(m-k) \cdot p_{A_{i+1}}(k)+\sum_{(1-\beta) 2^n +1\leq k \leq m} p_{B}(m-k) \cdot p_{A_{i+1}}(k)\\
                           &=p_{B}(m)+\sum_{(1-\beta) 2^n+1 \leq k \leq m} p_{B}(m-k) \cdot p_{A_{i+1}}(k)\;.
\end{align*}
For all $k\leq m\leq 2^n$ any partition of $k$ with all parts in $A_{i+1}$ uses at most ${1/(1-\beta)}$ integers from $A_{i+1}\cap [m]$. Therefore for all $k\in [(1-\beta) 2^n+1,m]$ we have
\[
    p_{A_{i+1}}(k)\leq m^{1/(1-\beta)}\;.
\]
Next, if $k>(1-\beta) 2^n$ then
\[
    \log_{2}(k)+\log_2(1/(1-\beta))\geq \log_{2}(2^n) = n\geq |B|\;.
\]
Therefore by Claim \ref{Claim - trivial bound} we obtain that for all $k\in [(1-\beta) 2^n+1,m]$ we have
\[
    p_{B}(k)\leq (k+1)^{|B|}\leq (m+1)^{\log_2(k)+\log_2(1/(1-\beta))}\;.
\]
Combining the above two observations we obtain,
\begin{align*}
    p_A(m) &\leq \sum_{0\leq k \leq m} (m+1)^{\log_2(k)+\log_2(1/(1-\beta))+1/(1-\beta)}\leq (m+1)(m+1)^{\log_2(m)+4}\leq e^{o(\sqrt{m})}\;.\qedhere
\end{align*}
\end{casesp}

\end{proof}

\end{document}